\documentclass[11pt,reqno]{amsart}
\usepackage{amsmath}
\usepackage{amssymb}
\usepackage{amsthm}
\usepackage{enumerate}
\usepackage[usenames]{color}
\usepackage{eepic,epic}
\usepackage{url} 
\usepackage{booktabs,calc}
\usepackage{epstopdf}

\usepackage{graphicx}

\newtheorem{thm}{Theorem}[section]
\newtheorem{cor}[thm]{Corollary}
\newtheorem{lem}[thm]{Lemma}
\newtheorem{prop}[thm]{Proposition}

\theoremstyle{definition}
\newtheorem{defn}[thm]{Definition}

\theoremstyle{remark}
\newtheorem{rem}[thm]{Remark}
\theoremstyle{definition}
\newtheorem{exm}[thm]{Example}
\numberwithin{equation}{section}

\begin{document}
\title[]
{A unified framework for distributed optimization algorithms over time-varying directed graphs}

\author{Woocheol Choi}
\address{Department of Mathematics, Sungkyunkwan University, Suwon 440-746, Republic of Korea}
\email{choiwc@skku.edu}

\author{Doheon Kim}
\address{School of Mathematics, Korea Institute for Advanced Study, Seoul 02455, Republic of Korea}
\email{doheonkim@kias.re.kr}

\author{Seok-Bae Yun}
\address{Department of Mathematics, Sungkyunkwan University, Suwon 440-746, Republic of Korea}
\email{sbyun01@skku.edu}



\maketitle

\begin{abstract}
In this paper, we propose a framework under which  the decentralized optimization algorithms suggested in \cite{JKJJ,MA, NO,NO2} can be treated in a unified manner. More precisely, we show that the distributed subgradient descent algorithms \cite{JKJJ, NO}, the subgradient-push algorithm \cite{NO2}, and the distributed algorithm with row-stochastic matrix  \cite{MA} can be derived by making suitable choices of consensus matrices, step-size and subgradient from the
decentralized subgradient descent proposed in \cite{NO}. As a result of such unified understanding, we provide a convergence proof that covers
the algorithms in \cite{JKJJ,MA, NO,NO2} under a novel algebraic condition that is strictly weaker than the conventional graph-theoretic condition in \cite{NO}. 
This unification also enables us to derive a new distributed optimization scheme.

\end{abstract}

\section{Introduction}
In this paper, we consider $N$ agents cooperating with each other to solve the following optimization problem:
\begin{equation}\label{opt}
\mbox{Find a minimizer}\quad x^* \in \mathbb{R}^d \quad\mbox{of the function} \quad   f( x):=   \sum_{i=1}^{N} f_i( x).
\end{equation}
Here, for each $1\leq i\leq N$, the function $f_i : \mathbb{R}^d \rightarrow \mathbb{R}$ is local cost known only to the $i$-th agent. The $N$ agents are connected by a network and each agent can receive information from its neighboring agents. Distributed optimization has received a lot of attention due to its application for various problems containing wireless sensor network, multi-agent control, machine learning. In distributed optimization, many agents have their own local cost and try to find a minimizer of the sum of those local cost functions in a collaborative way. The algorithms consist of   optimization step for local function and consensus step by communication.

 In the seminal work \cite{NO}, Nedic et al introduced the concept of the decentralized (sub)gradient descent (DGD) to solve \eqref{opt}:
\begin{equation}\label{frame}
 x_i(t+1)=\sum_{j=1}^N p_{ij}(t)x_j(t)-\delta_i(t)g_i(t),\quad i=1,\dots,N,\quad t=0,1,2,\dots,
\end{equation}
with the $N\times N$ row-stochastic matrices $P(t):=(p_{ij}(t))$, the step-size $\delta_i(t)$, and the subgradients $g_i(t)\in\partial f_i(x_i(t))$ $(1\leq i\leq N,~t=0,1,2,\dots)$. The matrix $P$ contains the information on the connectivity of the underlying network: $p_{ij}(t)$ vanishes if the $i$-th agent does not receive any information from the $j$-th agent at time $t$. Otherwise $p_{ij}$ is always positive. 

Many varaints and extensions of \eqref{frame} were proposed in the literature, and their convergence properties were extensively studied. In \cite{JKJJ}, the authors designed a distributed subgradient algorithm to solve \eqref{opt} given by
\begin{equation}\label{eq-1-1}
x_i (t+1) = \sum_{j=1}^N p_{ij} (t) \Big( x_j (t) - \delta_j (t) g_j (t)\Big),
\end{equation}
The convergence of the above algorithms were investigated for undirected graph under the assumption that $\{p_{ij}(t)\}$ is a symmetric and doubly-stochastic matrix. 

On the other hand, it is more practical to consider directed communications between agents under  certain environments. The network of agents is then represented by a directed graph.  To handle the distributed optimization on directed graph, Nedic-Olshevsky \cite{NO2} employed the push-sum protocol \cite{KDG} to design the subgradient-push method:
\begin{equation}\label{eq-1-2}
\begin{cases}
\displaystyle   w_i(t+1)=\sum_{j=1}^Na_{ij}(t)  x_j(t) \in\mathbb R^d,\quad i=1,\dots,N,\quad t=0,1,2,\dots,\\
\displaystyle  y_i(t+1)=\sum_{j=1}^Na_{ij}(t)  y_j(t) \in\mathbb R_+,\\
\displaystyle  z_i(t+1)=\frac{  w_i(t+1)}{  y_i(t+1)}\in\mathbb R^d,\\
\displaystyle   x_i(t+1)= w_i(t+1)- \delta_i(t+1)g_i(t+1),\quad g_i(t+1)\in\partial f_i(z_i(t+1)),
\end{cases}
\end{equation}
where $\{a_{ij}(t)\}$ is a column-stochastic matrix   used for the direct communication. 
When stated in terms of $w_i (t)$ and $y_i (t)$, \eqref{eq-1-2} can be reformulated in a more succinct form:
\begin{equation} 
\begin{cases}
\displaystyle   w_i(t+1)=\sum_{j=1}^Na_{ij}(t)  (w_j (t) -\delta_j (t) g_j (t)),\quad g_i (t) \in \partial f_i \left( \frac{w_i (t)}{y_i (t)}\right) \\
\displaystyle  y_i(t+1)=\sum_{j=1}^Na_{ij}(t)  y_j(t),\quad 1 \leq i \leq N, \quad t \geq 0.
\end{cases}
\end{equation}
This algorithm has been extended to various problems containing the stochastic distributed optimization \cite{NO3, TAHR}, online distributed optimization \cite{AGL}.

Recently the authors in \cite{MA} designed a distributed subgradient algorithm for directed graph by using row-stochastic matrix for communication
\begin{equation}\label{eq-1-3}
\begin{cases}
\displaystyle     x_i(t+1)=\sum_{j=1}^N  p_{ij}  x_j(t)- \frac{\delta(t)}{\bold e_i^\top z_{i}(t)}  g_i(t),\quad  g_i(t)\in\partial f_i(  x_i(t)),\\
\displaystyle  z_i(t+1)=\sum_{j=1}^N  p_{ij}  z_j(t),\quad z_i(0)=\bold e_i,\quad 1\leq i\leq N,\quad t\geq0. 
\end{cases}
\end{equation}
We also refer to \cite{MA2} for further details and the convergence analysis of the above algorithm.

The main goal of this paper is two-fold. First, we show that the algorithms \eqref{eq-1-1} - \eqref{eq-1-3} can be exploited in a unified way in the framework \eqref{frame}. More pecisely, we show that 
the above mentioned algorithms \eqref{eq-1-1}-\eqref{eq-1-3} can all be derived from \eqref{frame}
by making specific choices of the row-stochastic matrix $P(t)$ and the scalars $\delta_i (t)$ in \eqref{frame}. This 
unified perspective also enables us to design a new algorithm which resembles the subgradient-push method \cite{NO2}, but differs in the order of the consensus step and the gradient descent step:
\begin{equation}\label{eq-1-4}
\begin{cases}
\displaystyle  w_i(t+1)=\sum_{j=1}^N a_{ij}(t)w_j(t)-\theta_i(t)g_i(t),\quad g_i(t)\in\partial f_i\left(\frac{w_i(t)}{y_i(t)}\right),
\\
\displaystyle  y_i(t+1)=\sum_{j=1}^N a_{ij}(t)y_j(t),\quad1\leq i\leq N, \quad t\geq0,
\end{cases}
\end{equation}
where $A(t) = \{a_{ij}(t)\}$  are column stochastic. We refer to \cite{AS, Jakovetic, SHL, XPNK, XZSX, ZYGY} for unifications of gradient tracking type algorithms.


Secondly, as a result of such unified understanding of the above schemes, we provide a convergence proof that covers the optimization schemes \eqref{eq-1-1} - \eqref{eq-1-4} in a unified manner. The novelty of this convergence proof, aside from the unification itself, lies in providing the convergence of the distributed algorithm \eqref{frame} for which the matrices $P(t)$ are row-stochastic, but not necessarily column-stochastic. While the convergence analysis of \eqref{frame} is well-understood when $P(t)$ are doubly stochastic (both row-stochastic and column-stochastic) and the stepsize $\delta_i (t)$ satisfies a sutiable decaying property and is independent of $i$, i.e., $\delta_i (t) = \delta (t)$,  the convergence analysis of \eqref{frame}  for general row-stochastic  $P(t)$ is, to the best knowledge of the authors, not established yet. Moreover, when $P(t)$ is only row-stochastic and $\delta_i (t) = \delta (t)$, the distributed algorithm \eqref{frame} may fail to converge to an optimizer of the problem even though a suitable decaying property is given. 

We establish the convergence of \eqref{frame} to a minimizer for row-stochastic $P(t)$ under a general condition on the stepsize $\delta_i (t)$ which is associated to the sequence of absolute probability vectors \cite{Kolmogoroff}. We mention that the sequence of absolute probability vectors has been used importantly in the literature containing consensus algorithms \cite{NL} and distributed optimization algorithms \cite{SXK}.

 Another contribution of our convergence result is 
 that we successfully replace the conventional graph theoretic condition \cite{NO2} on the underlying graph with an algebraic condition on the consensus matrix in the convergence proof, and show that the latter  is strictly more general than the former. 
 More precisely, the previous works assumed that each vertex of graph $G(t)$ has a self-loop, and finite unions of time-varying graphs, 
\[
\bigcup_{t\leq s< t+T}G(s):=\left(V,\bigcup_{t\leq s< t+T}E(s)\right)\quad (\forall~t\geq0)  \hspace{0.6cm} \mbox{for some fixed }T\in\mathbb{N},
\]
are strongly connected. We replace these conditions with a strictly weaker non-vaninishing condition on the time-varying consensus matrices up to a finite product:
\[
 P(t+T,t):= P(t+T-1)P(t+T-2) \cdots P(t) >0\quad (\forall~t\geq0) \hspace{0,6cm} \mbox{for some fixed }T\in\mathbb{N}.
\]
We also provide an explicit example which satisfies the latter but not the former (see Section 4). Here $P(t+T, t)>0$ means that every element of $P(t+T, t)$ is positive. We note that our condition does not require the existence of self-loop at each vertex of graph $G(t)$, which was necessary in most of the previous works.


{\bf Notation.} Before we finish this introduction we set up notational convention that is kept throughout the paper.
\begin{itemize}
\item For $1\leq p\leq\infty$, $v=(v_1,\dots,v_d)^\top \in \mathbb R^{d \times1}$ and $A=(a_{ij})\in\mathbb R^{m\times n}$, define
\[
\|v\|_p:=\left(\sum_{l=1}^d|v_l|^p\right)^\frac{1}{p}\quad(1\leq p<\infty),\quad \|v\|_\infty=\max_{1\leq l\leq d}|v_l|,\quad \|A\|_p:=\sup_{x\neq0}\frac{\|Ax\|_p}{\|x\|_p}\quad(1\leq p\leq \infty).
\]
Note that
\[
\|v\|_2\leq \|v\|_1,\quad \|A\|_2=\|A^\top\|_2,\quad \|A\|_\infty=\max_{1\leq i\leq n}\sum_{j=1}^n|a_{ij}|,\quad \|A\|_1=\max_{1\leq j\leq n}\sum_{i=1}^n|a_{ij}|.
\]
\item Unless a specification is needed, we use $\|\cdot\|$ generically to denote any sub-multiplicative matrix norm throughout the paper: If we write a statement with this norm, then it means that it holds for any sub-multiplicative matrix norm since they are all equivalent. 
\item Lastly, we denote by $I_N$, $O_N$ and $1_{N}$ the $N\times N$ identity matrix, the $N\times N$ zero matrix and the $N\times1$ matrix whose entries are all 1, respectively.
\end{itemize}
The paper is organized as follows.  In Section 2, we derive the algorithms \eqref{eq-1-1}-\eqref{eq-1-4} from the distributed algorithm \eqref{frame} choosing suitable parameters. In Section 3, we review some fundamental properties of row-stochastic matrices. In Section 4, we state the convergence result for \eqref{frame} and discuss the assumptions used in the convergence result. In Section 5, we apply the result of Section 4 to derive the convergence estimates of the algorithms \eqref{eq-1-1}-\eqref{eq-1-4}. Section 6 is devoted to give the proof of the convergence result stated in Section 4. 

\section{Derivation of distributed algorithms from \eqref{frame}}
In this section, we show that  the distributed algorithms \eqref{eq-1-1}-\eqref{eq-1-3} 
can be recovered from \eqref{frame} by making specific choices of $P(t)$ and $\delta_i (t)$. For this purpose, we write \eqref{frame} as
\begin{equation}\label{frame2}
 X(t+1)=P(t)X(t)-\Delta(t)G(t),\quad t\geq0,
\end{equation}
with the following notations:
\begin{equation}
\begin{array}{ll}
X(t)=[x_1(t)\dots x_N(t)]^\top\in\mathbb R^{N\times d},\quad &P(t) =(p_{ij}(t))\in\mathbb R^{N\times N},
\\
\Delta(t) =\operatorname{diag}(\delta_1(t),\dots,\delta_N(t))\in\mathbb R^{N\times N},\quad  &G(t) =[g_1(t)\dots g_N(t)]^\top\in\mathbb R^{N\times d}.
\end{array}
\end{equation}

\begin{exm}\label{E3.1} {\bf Derivation of \eqref{eq-1-1}}:
	We divide \eqref{frame2} in odd and even cases:
	\begin{align}\label{eo}
	\begin{split}
  X(2t+1)&=P(2t)X(2t)-\Delta(2t)G(2t),	\\
   X(2t+2)&=P(2t+1)X(2t+1)-\Delta(2t+1)G(2t+1).
  \end{split}
	\end{align}
We choose $P(2t)\equiv I_N$ in $\eqref{eo}_1$ to get
\begin{align}\label{eo1}
\begin{split}
X(2t+1)&=X(2t)-\Delta(2t)G(2t).
\end{split}
\end{align}
On the other hand, we set $\Delta(2t+1)\equiv O_N$ to reduce $\eqref{eo}_2$ into
	\begin{align}\label{eo2}
\begin{split}
X(2t+2)&=P(2t+1)X(2t+1).
\end{split}
\end{align}
Inserting \eqref{eo1} into \eqref{eo2}, we obtain
\begin{align*}
\begin{split}
X(2t+2)&=P(2t+1)\big(X(2t)-\Delta(2t)G(2t)\big).
\end{split}
\end{align*}
Finally, we make the following choices of $P(t)$, $\Delta(t)$, and $G(t)$: 
\begin{align*}
&\tilde P(t)=(\tilde p_{ij}(t)):=P(2t+1),\quad \tilde\Delta(t)=\operatorname{diag}(\tilde\delta_1(t),\dots,\tilde\delta_N(t)):=\Delta(2t),\\
&\hspace{3cm} \tilde G(t):=[\tilde g_1(t)\dots \tilde g_N(t)]^\top:=G(2t).
\end{align*}
and introduce the new optimizing variable:
\begin{equation*}
\tilde X(t)=[\tilde x_1(t)\dots \tilde x_N(t)]^\top:=X(2t),
\end{equation*}
to obtain the distributed algorithm \eqref{eq-1-1}:
\[
\tilde X(t+1)=\tilde P(t)\big(\tilde X(t)-\tilde\Delta(t)\tilde G(t)\big),\quad t\geq0,
\]
or equivalently,
\begin{equation*}
  \tilde x_i(t+1)=\sum_{j=1}^N\tilde p_{ij}(t) (\tilde x_j(t)- \delta_j(t)\tilde g_i(t)),\quad \tilde g_i(t)\in\partial f_i(\tilde x_i(t)),\quad 1\leq i\leq N,\quad t\geq0. 
\end{equation*}
\end{exm}

\begin{exm} \label{E3.2} {\bf Derivation of \eqref{eq-1-3}:}
We let the matrix $P(t)$ of  \eqref{frame2} independent  of time, i.e., 
\[
P(0)=P(1)=P(2)=\dots=:P,
\]
and introduce a new variable $Z(t)$ defined as the $t$-th power of $P$ with the convention $P^0:=I_N$:
\[
Z(t)=[z_1(t)\dots z_N(t)]^\top:=P^{t},\quad\mbox{with}\quad z_i(t)=[z_{i1}(t)\dots z_{iN}(t)]^\top,\quad 1\leq i\leq N,\quad t\geq0.
\]
Then, we have
\begin{align}\label{Z}
Z(t+1)=P^{t+1}=PP^t=PZ(t).
\end{align}
We then choose our new step-size by
\[
\delta_i(t)\equiv\frac{\delta(t)}{\mbox{$i$-th diagonal entry of }P^{t}},\quad t\geq0,
\]
for some scalars $\delta(t)$ $(t\geq0)$, to get from \eqref{frame2} that
\begin{align}\label{X}
X(t+1)=PX(t)-\frac{\delta(t)}{P^t_{ii}}\,G(t).
\end{align}
From \eqref{Z} and \eqref{X}, we derive the distributed algorithm \eqref{eq-1-3}:
\begin{equation*}
\begin{cases}
\displaystyle     x_i(t+1)=\sum_{j=1}^N  p_{ij}  x_j(t)- \frac{\delta(t)}{z_{ii}(t)}  g_i(t),\quad  g_i(t)\in\partial f_i(  x_i(t)),\\
\displaystyle  z_i(t+1)=\sum_{j=1}^N  p_{ij}  z_j(t),\quad z_i(0)=\bold e_i,\quad 1\leq i\leq N,\quad t\geq0,
\end{cases}
\end{equation*}
where $\bold e_i\in\mathbb R^N$ denotes the $i$-th standard unit vector.
\end{exm}

\begin{exm}\label{E3.3} {\bf Derivation of a new algorithm:}
In this example, we propose a new algorithm that generalizes  the subgradient-push method in \cite{NO2} (see Example \ref{E3.4}). Let the sequence $\{A(t)\}_{t\geq0}$ of $N\times N$ column-stochastic matrices and the sequence $\{Y(t)\}_{t\geq0}$ of $N\times1$ matrices satisfying
\[
Y(t+1)=A(t)Y(t),\quad Y(t)>0,\quad t\geq0
\]
be given. Here we remark that one may replace the condition ``$Y(t)>0$ $(\forall~t\geq0)$'' by  a stronger condition ``each $A(t)$ $(t\geq0)$ has no zero row and $Y(0)>0$''. 
Indeed, we have 
\[
\min_{1\leq i\leq N}y_i(t+1)=\min_{1\leq i\leq N}\left(\sum_{j=1}^N a_{ij}(t)y_j(t)\right)\geq \min_{1\leq i\leq N}\left(\sum_{j=1}^N a_{ij}(t)\right)\cdot\left(\min_{1\leq k\leq N}y_k(t)\right),
\]
and by the assumption that $A(t)$ has no zero row, we have
\[
\min_{1\leq i\leq N}\left(\sum_{j=1}^N a_{ij}(t)\right)>0.
\]
 So
we can prove $Y(t)>0$ by inducting on $t$.

We then make the following choice for $\{P(t)\}_{t\geq0}$:
\[
P(t)\equiv \operatorname{diag}(Y(t+1))^{-1} A(t)\operatorname{diag}(Y(t)),\quad t\geq0,
\]
which is row-stochastic by construction:
\begin{equation}\label{eq-3-1}
\begin{split}
P(t)1_{N}&=\operatorname{diag}(Y(t+1))^{-1} A(t)\operatorname{diag}(Y(t))1_{N}\\
&=\operatorname{diag}(Y(t+1))^{-1} A(t) Y(t) \\
&=\operatorname{diag}(Y(t+1))^{-1}   Y(t+1) \\
&=1_{N}.
\end{split}
\end{equation}
We choose this $P(t)$ in \eqref{frame2} and, finally we multiply  $\operatorname{diag}(Y(t+1))$ to the left of both sides of \eqref{frame2}
\begin{equation*}
\operatorname{diag}(Y(t+1))X(t+1)=A(t)\operatorname{diag}(Y(t))X(t)-\operatorname{diag}(Y(t+1))\Delta(t)G(t),\quad t\geq0,
\end{equation*}
 and introduce
\[
W(t) :=\operatorname{diag}(Y(t))X(t),\quad \Theta(t):=\operatorname{diag}(Y(t+1))\Delta(t),\quad   t\geq0
\]
to get
\begin{equation}\label{choen}
\begin{cases}
\displaystyle  Y(t+1)=A(t)Y(t), \\
\displaystyle  W(t+1)=A(t)W(t)-\Theta(t)G(t),  \quad t\geq0.
\end{cases}
\end{equation}
If we denote
\begin{equation}
\begin{array}{ll}
W(t)=[w_1(t)\dots w_N(t)]^\top,\quad & Y(t)=[y_1(t)\dots y_N(t)]^\top,\quad
\\
\Theta(t)=\operatorname{diag}(\theta_1(t)\dots\theta_N(t)),\quad &A(t)=(a_{ij}(t)),
\end{array}
\end{equation}
then we can rewrite \eqref{choen} as
\begin{equation}\label{algorithm3}
\begin{cases}
\displaystyle  y_i(t+1)=\sum_{j=1}^N a_{ij}(t)y_j(t),\quad1\leq i\leq N, \quad t\geq0, \\
\displaystyle  w_i(t+1)=\sum_{j=1}^N a_{ij}(t)w_j(t)-\theta_i(t)g_i(t),\quad g_i(t)\in\partial f_i\left(\frac{w_i(t)}{y_i(t)}\right).
\end{cases}
\end{equation}

\end{exm}

\begin{exm}\label{E3.4}{\bf Derivation of \eqref{eq-1-2}}
As in the case of Example \ref{E3.1}, we divide \eqref{choen} in the Example \ref{E3.3} into odd and even cases as 
\begin{align}\label{eo3}
	\begin{split}
  Y(2t+1)&=A(2t)Y(2t)	\\
   W(2t+1)&=A(2t)W(2t)-\Theta(2t)G(2t) 
   \\
  Y(2t+2)&=A(2t+1)Y(2t+1)	\\
   W(2t+2)&=A(2t+1)W(2t+1)-\Theta(2t+1)G(2t+1)
  \end{split}
	\end{align}
and set
\[
A(2t)=I_N,\quad \Delta(2t+1)=O_N,\quad t\geq0.
\]
It leads to
\begin{align}\label{eo3}
	\begin{split}
  Y(2t+1)&= Y(2t)	\\
   W(2t+1)&= W(2t)-\Theta(2t)G(2t)  
   \\
  Y(2t+2)&= A(2t+1)Y(2t+1)	\\
   W(2t+2)&=A(2t+1) W(2t+1),
  \end{split}
	\end{align}
	which is reduced to
\begin{align}
	\begin{split}
  Y(2t+2)&= A(2t+1)Y(2t)	\\
   W(2t+2)&= A(2t+1)(W(2t)-\Theta(2t)G(2t)).
  \end{split}
	\end{align}
This, together with the following choices:
\begin{equation}
\begin{array}{ll}
\hat X(t)=[\hat x_1(t)\dots \hat x_N(t)]^\top:=W(2t+1),\quad & 
\hat G(t):=[\hat g_1(t)\dots \hat g_N(t)]^\top:=G(2t),
\\
\hat Y(t)=[\hat y_1(t)\dots \hat y_N(t)]^\top:=Y(2t),\quad & \hat A(t)=(\hat a_{ij}(t)):=A(2t+1),
\\
\hat\Theta(t)=\operatorname{diag}(\hat\theta_1(t),\dots,\hat\theta_N(t)):=\Theta(2t),
\quad& \hat W(t)=[\hat w_1(t)\dots \hat w_N(t)]^\top:=W(2t),
\end{array}
\end{equation}
lead to the subgradient-push method in \cite{NO2} 
\begin{align}
	\begin{split}
  \hat{Y}(t+1)&= \hat{A} (t) \hat{Y}(t)	\\
   \hat{W}(t+1)&= \hat{A}(t)(\hat{W}(t)-\hat{\Theta}(t)\hat{G}(t)),
  \end{split}
	\end{align}
which can be written as
\begin{equation*}
\begin{cases}
\displaystyle\hat  y_i(t+1)=\sum_{j=1}^N\hat a_{ij}(t)\hat  y_j(t) \in\mathbb R_+,\\
\displaystyle {\hat  x_i(t)=\hat  w_i(t)- \hat \theta_i(t)\hat g_i(t),\quad \hat  g_i(t)\in\partial f_i\left(\frac{\hat  w_i(t)}{\hat y_i(t)}\right),}\\
\displaystyle \hat  w_i(t+1)=\sum_{j=1}^N\hat a_{ij}(t)\hat  x_j(t) \in\mathbb R^d.
\end{cases}
\end{equation*}
In Table 1, we provide a systematic summary of this section.
\end{exm}
\begin{table}\label{tab}
\centering
\caption{\bf Summary of the derivations of optimization algorithms}
  \begin{tabular}{ | p{6cm} | p{7cm} | p{1.3cm} | }
   \hline
   Choice of $P(t)$, $\delta_i (t)$, $x_i(t)$ in \eqref{frame}& Corresponding algorithm & Ref. \\[0.1cm]
   \hline
     \smallskip 
     \vtop{\hbox{\strut      $P (t) =P(t)$}\hbox{\strut$\delta_i(t)=\delta(t)$}\hbox{\strut  $x_i (t) = x_i (t)$ }}     &     \smallskip \vtop{\hbox{${x}_i (t+1) = \sum_{j=1}^N {p}_{ij} (t) {x}_j (t) - \delta (t) g_i(t)$}\hbox{$  g_i(t)\in\partial f_i\left(x_i(t)\right)$}\hbox{${P}$: doubly-stochastic}} &      \smallskip \cite{NO} \\[1cm]
    \hline 
         \smallskip
     \vtop{\hbox{\strut      $P (2t) =I_N$, $P(2t+1) = \tilde{P}(t)$} \hbox{\strut $\delta_i (2t+1) =0$, $\delta_i (2t) = \tilde\delta_i (t)$ }\hbox{\strut  $x_i (2t) = \tilde{x}_i (t)$ }\hbox{\strut \medskip $x_{i} (2t+1) = \tilde{x}_i (t) - \tilde\delta_i (t) \tilde{g}_i (t)$ }}
     &      \smallskip \vtop{\hbox{$\tilde{x}_i (t+1) = \sum_{j=1}^N \tilde{p}_{ij} (t) \Big(\tilde{x}_j (t) - \tilde\delta_{j} (t) \tilde g_j(t)\Big)$} \hbox{$ \tilde g_i(t)\in\partial f_i\left(\tilde x_i(t)\right)$}\hbox{$\tilde{P}$: doubly-stochastic}} &     \smallskip \cite{JKJJ} \\[2cm]
    \hline      \smallskip
  \vtop{\hbox{\strut $P(t) \equiv P$}\hbox{\strut $\delta_i (t) = {\delta (t)}/{z_{ii}(t)}$}}\hbox{$x_i (t) = x_i (t)$} &     \smallskip \vtop{\hbox{\strut  $x_i (t+1) = \sum_{j=1}^N p_{ij} x_j (t) - \frac{\delta(t)}{z_{ii}(t)} g_i(t)$} \hbox{\strut $z_{i}(t+1) = \sum_{j=1}^N p_{ij}z_j (t)$}\hbox{$  g_i(t)\in\partial f_i\left(x_i(t)\right)$}\hbox{${P}$: row-stochastic}} &      \smallskip\cite{MA} \\[2cm]
   \hline      \smallskip
   \vtop{\hbox{\strut $P(t) \equiv \textrm{diag}(Y(t+1))^{-1} A(t) \textrm{diag}(Y(t))$}\hbox{\strut $\delta_i (t) = {\theta_i (t)}/{y_{i}(t)}$}} \hbox{$x_i (t) = w_i (t)/y_i (t)$}&     \smallskip \vtop{\hbox{\strut $w_i (t+1) = \sum_{j=1}^N a_{ij} (t) w_j (t) - \theta_i (t) g_i(t)$}\hbox{\strut $y_i (t+1) =\sum_{j=1}^N a_{ij} (t) y_j (t)$}\hbox{$  g_i(t)\in\partial f_i\left(\frac{   w_i(t )}{  y_i(t )}\right)$}
\hbox{$A$: column-stochastic}} &     \smallskip This paper \\[1cm]
   \hline      \smallskip 
   \vtop{\hbox{\strut $P(2t)=I_N$} \hbox{\strut $P(2t+1) \equiv \textrm{diag}(\hat Y(t+1))^{-1} \hat A(t) \textrm{diag}(\hat Y(t))$}\hbox{\strut $\delta_i (2t) = {\hat{\theta}_i (t)}/{\hat{y}_{i}(t)}$, $\delta_i (2t+1)=0$}}  \hbox{$x_i (2t) = \hat{w}_i (t)/\hat{y}_i (t)$}\hbox{$x_i (2t+1) = \frac{\hat{w}_i (t)}{\hat{y}_i (t)}- \frac{\hat{\theta}_i (t)}{\hat{y}_{i}(t)} \hat g_i(t)$}&     \smallskip \vtop{\hbox{\strut $\hat{w}_i (t+1) = \sum_{j=1}^N \hat{a}_{ij} (t) \hat{x}_j (t)$}\hbox{\strut $\hat{y}_i (t+1) =\sum_{j=1}^N \hat{a}_{ij} (t) \hat{y}_j (t)$}\hbox{$\hat{x}_i (t+1) = \hat{w}_i (t+1) - \hat{\theta}_i \hat g_i(t+1)$} \hbox{$\hat  g_i(t+1)\in\partial f_i\left(\frac{\hat  w_i(t+1)}{\hat y_i(t+1)}\right)$}\hbox{$A$: column-stochastic}} &     \smallskip\cite{NO2} \\[1cm]
   \hline
  \end{tabular}
\centering
  \end{table}

\section{Properties of row-stochastic matrices}

In this section, we study some properties of ergodic sequences of row-stochastic matrices, which is crucially used in the sequel. 
%

\begin{defn}\mbox{~}
\begin{enumerate}
\item A matrix $A$ is \textit{non-negative (positive)} if all of its entries are  non-negative (positive), and we write $A\geq0~(>0)$.
\item A non-negative square matrix is \textit{row-stochastic} if all of its row has sum equal to $1$.
\item A non-negative vector is a \textit{probability vector} if the sum of its entries is equal to $1$.
\item A sequence $\{P(t)\}_{t\geq0}$  of row-stochastic matrices is \textit{ergodic} if there exists a sequence $\{v(t)\}_{t\geq0}$ of probability vectors satisfying
\[
\lim_{t\to\infty}P(t)P(t-1)\dots P(t_0)=1_{N}v(t_0)^\top,\quad \forall~t_0\geq0.
\]
\item A sequence $\{\pi(t)\}_{t\geq0}$ of probability vectors is a \textit{set of absolute probability vectors} for the sequence $\{P(t)\}_{t\geq0}$  of row stochastic matrices if
\begin{equation}\label{eq-2-1}
\pi(t+1)^\top P(t)=\pi(t)^\top,\quad t\geq0.
\end{equation}
\end{enumerate}
\end{defn}

The following proposition establishes a relation between ergodic sequences and sets of absolute probability vectors.
\begin{prop}[\cite{Kolmogoroff}]\label{P2.2}
Let $\{P(t)\}_{t\geq0}$ be a sequence of $N\times N$ row stochastic matrices, and let $\{\pi(t)\}_{t\geq0}$ be a sequence of $N\times1$ probability vectors. Then the following are equivalent.
\begin{enumerate}
\item $\{P(t)\}_{t\geq0}$ is ergodic and satisfies
\[
\lim_{t\to\infty}P(t)P(t-1)\dots P(t_0)=1_{N}\pi(t_0)^\top,\quad \forall~t_0\geq0.
\]
\item $\{\pi(t)\}_{t\geq0}$ is a unique set of absolute probability vectors for $\{P(t)\}_{t\geq0}$.
\end{enumerate}
\end{prop}
Next, we introduce a useful tool to study ergodicity of sequences of row-stochastic matrices.



\begin{defn}[\cite{Dobrushin1, Dobrushin2}]
For each $N\times N$ row-stochastic matrix $P=(p_{ij})$, we define its \textit{ergodicity coefficient}   by
\begin{equation}\label{eq-2-10}
\tau(P):=\frac{1}{2}\max_{i_1,i_2} \sum_{j=1}^N|p_{i_1j}-p_{i_2j}|= 1-\min_{i_1,i_2} \sum_{j=1}^N\min\{p_{i_1j},p_{i_2j}\}.
\end{equation}
\end{defn}
The identity \eqref{eq-2-10} can be derived in the following way:
\begin{align*}
\frac{1}{2}\max_{i_1,i_2} &\sum_{j=1}^N|p_{i_1j}-p_{i_2j}|=\frac{1}{2}\max_{i_1,i_2} \sum_{j=1}^N\left(p_{i_1j}+p_{i_2j}-2\min\{p_{i_1j},p_{i_2j}\}\right)\\
&=\frac{1}{2}\max_{i_1,i_2}\left(2-2 \sum_{j=1}^N \min\{p_{i_1j},p_{i_2j}\}\right) =1-\min_{i_1,i_2} \sum_{j=1}^N\min\{p_{i_1j},p_{i_2j}\}.
\end{align*} 

For any $N\times N$ row-stochastic matrix $P$, it is clear from the definition that $0\leq \tau(P)\leq 1$, and that
\[
\tau(P)=0 \quad\mbox{if and only if}\quad P=1_{N}v^\top\quad\mbox{for some}\quad v\in\mathbb R^{N\times1},
\]
which reveals a close relationship between $\tau$ and ergodicity.
We finish this section by introducing some useful properties of $\tau$.

\begin{lem}[\cite{Dobrushin1, Dobrushin2}]\label{L2.4}
\begin{enumerate}\mbox{~}
\item
For any row-stochastic matrix $P$, we have
\[
\tau(P)=\sup_{\substack{u\neq0,\\u^\top 1_{N}=0}}\frac{\|P^\top u\|_1}{\|u\|_1}
\]
\item
For any row-stochastic matrices $P_1,P_2$ of the same size, we have
\[
\tau(P_1P_2)\leq \tau(P_1)\tau(P_2).
\]
\end{enumerate}
\end{lem}

\begin{lem}[\cite{CS}]\label{L2.5}
A sequence $\{P(t)\}_{t\geq0}$ of $N\times N$ row stochastic matrices is ergodic if and only if
\[
\lim_{t\to\infty}\tau(P(t)P(t-1)\dots P(t_0))=0,\quad \forall~t_0\geq0.
\]
\end{lem}
\section{Conditions for convergence to an optimum}
In this section, we present the main theorem of this paper, which states a general sufficient condition for which  the algorithm \eqref{frame2} finds a minimizer $x^*$ of the distributed optimization problem \eqref{opt}. 

For our convergence result, we make the following assumptions for $f_i$:
\begin{enumerate}
\item For each $i=1,\dots,N$, $f_i$ is convex, and we have $L_i:=\sup\limits_{z\in \mathbb R^d}\sup\limits_{g\in\partial f_i(z)}\|g\|_1<\infty.$
\item There is at least one solution to the minimization problem $x^*\in\operatorname{argmin}_{x\in\mathbb R^d}f(x)$.
\end{enumerate}
From now on, to the sequence $\{P(t)\}_{t\geq0}$, we associate the backward products
\[
P(t,t_0):=P(t-1)P(t-2)\dots P(t_0),\quad t\geq t_0\geq0
\] 
with the convention $P(t_0,t_0):=I_N$ $(t_0\geq0)$.
Below, we introduce the main assumptions which constitute a sufficient condition to guarantee the convergence of the algorithm \eqref{frame2} to a minimizer of $f$.
\begin{itemize}
\item ($\mathcal A1$): The sequence $\{P(t)\}_{t\geq0}$ of $N\times N$ row-stochastic matrices satisfies
\begin{equation}\label{eq-4-1}
p^+:=\inf_{t\geq0}\left[\min_{i,j:~p_{ij}(t)>0}p_{ij}(t)\right]>0
\end{equation}
and there exists $T \in \mathbb{N}$ such that
\[
P(t+T,t)>0\quad \forall~t\geq0.
\]

\item ($\mathcal A2$): Each $\delta_i(\cdot)$ is nonnegative,  and $\sum_{t=0}^\infty \left(\|\Delta(t)\|\sup_{\ell\geq t}\|\Delta(\ell)\|\right)<\infty.$
\item ($\mathcal A3$): The set of absolute probability vectors for $\{P(t)\}_{t\geq0}$ denoted by
\[
\pi(t)=(\pi_1(t),\dots,\pi_N(t))^\top\quad (t\geq0)
\]
satisfies
$\sum\limits_{t=0}^\infty \|\Delta(t)\pi(t+1)\|=\infty$ and $\sum\limits_{t=0}^\infty\sqrt{t}\left(\max\limits_{i,j}|\pi_i(t+1)\delta_i(t)-\pi_j(t+1)\delta_j(t)|\right)<\infty$.
\end{itemize}
We remark that the existence and uniqueness of the vector $\pi (t)$ in $(\mathcal A3)$ is guaranteed by ($\mathcal A1$) as proved in  Lemma 4.2 below.

Now we state the main theorem of this paper.
\begin{thm}\label{Tmain}
Let $X(t)$ be a solution to \eqref{frame2}. Suppose that $(\mathcal A 1)$-$(\mathcal A 3)$ hold. Then there exists some minimizer $x^*$ of $f$ such that $\lim\limits_{t\to\infty}x_i(t)=x^*$ for all $1\leq i\leq N$. 
\end{thm}
In the following lemma, we prove that if there exists a  sequence $\{P(t)\}_{t\geq1}$ satisfying $(\mathcal A1)$ is given, then there exist sequences $\{\Delta(t)\}_{t\geq0}$ and $\{\pi(t)\}_{t\geq0}$ satisfying $(\mathcal A 2)$-$(\mathcal A 3)$.
\begin{lem}\label{L4.2}
Suppose that ($\mathcal A1$) holds. Then the following assertions hold.
\begin{enumerate}
\item For all $t\geq t_0\geq0$, the function $\tau$ defined in \eqref{eq-2-10} satisfies
\[
\tau(P(t,t_0))\leq C\lambda^{t-t_0}
\]
for some constants $C>0$ and $0<\lambda<1$, both  independent of $t$ and $t_0$.
\item The set $\{\pi(t)\}_{t\geq0}$ of absolute probability vectors defined in $(\mathcal A3)$ uniquely exists, and satisfies
\[
\pi_i(t)\geq(p^+)^T,\quad \forall~ 1\leq i\leq N,\quad\forall~t\geq 0.
\]
\item If we set
\begin{equation}\label{eq-4-2}
\delta_i(t):=\frac{c\bar\delta(t)}{\pi_i(t+1)+\epsilon_i(t)},\quad \forall~1\leq i\leq N,\quad~\forall~ t\geq0
\end{equation}
for some  $c>0$, $\bar\delta(t)\geq0$ satisfying $\sum_{t\geq0}\bar\delta(t)\sup_{\ell\geq t}\bar\delta(t)<\infty$ and $\sum_{t\geq0}\bar\delta(t)=\infty$, and $\epsilon_i(t)\in(-\pi_i(t+1),\infty)$ approaching zero geometrically fast as $t\to\infty$, then $\{\Delta(t)\}_{t\geq0}$ satisfies ($\mathcal A2$)-($\mathcal A 3$).
\end{enumerate}
\end{lem}
\begin{proof}
(1) For any row-stochastic $Q=(q_{ij})$, we deduce from \eqref{eq-2-1} the following inequality
\begin{equation}\label{eq-4-20}
\tau(Q)=1-\min_{i_1,i_2} \sum_{j=1}^N\min\{q_{i_1j},q_{i_2j}\}\leq 1-N\min_{i,j}q_{ij}.
\end{equation}
For any $s\geq0$, every entry of $P(s+T,s)$ is greater than or equal to $(p^+)^T$ defined in \eqref{eq-4-1}. Combining this with \eqref{eq-4-20} we deduce
\[
\tau(P(s+T,s))\leq 1-N(p^+)^T=:\gamma<1.
\]
Suppose $t-t_0=qT+m$ with $q\in\mathbb N\cup\{0\}$ and $0\leq m<T$. By Proposition 2.4, we have
\begin{align*}
\tau(P(t,t_0))&\leq \tau(P(t,t_0+qT))\prod_{k=1}^q\tau(P(t_0+kT,t_0+(k-1)T))\\
&\leq \gamma^{q}=\gamma^{\frac{t-t_0-m}{T} }=\gamma^{-\frac{m}{T}}\cdot \left(\gamma^{\frac{1}{T}}\right)^{t-t_0}<\gamma^{-1}\cdot \left(\gamma^{\frac{1}{T}}\right)^{t-t_0}=C\lambda^{t-t_0}.
\end{align*}
(2) By (1), Lemma \ref{L2.5} and Proposition \ref{P2.2}, the existence and uniqueness of $\{\pi(t)\}_{t\geq0}$ are guaranteed. By $(\mathcal A1)$, we have
\begin{align*}
\pi_i(t)&=\pi(t)^\top \bold e_i =\pi(t+T)^\top P(t+T,t)\bold e_i\geq\sum_{j=1}^N\pi_j(t+r_0)(p^+)^T=(p^+)^T.
\end{align*}
(3) By (2) and \eqref{eq-4-2}, there exist positive constants $c_1,c_2$ indepedent of $t$ such that
\[
c_1\bar\delta(t) \leq\delta_i(t)\leq c_2\bar\delta(t). 
\]
Combining this with the given assumptions, we have
\[
\sum_{t=0}^\infty \left(\|\Delta(t)\|\sup_{\ell\geq t}\|\Delta(\ell)\|\right)<\infty\quad\mbox{and}\quad \sum\limits_{t=0}^\infty \|\Delta(t)\pi(t+1)\|=\infty.
\]
Also note that $\max_{i,j}|\pi_i(t+1)\delta_i(t)-\pi_j(t+1)\delta_j(t)|$ approaches zero geometrically fast, since
\begin{align*}
|\pi_i(t+1)\delta_i(t)-\pi_j(t+1)\delta_j(t)|&=\frac{c\bar\delta(t)}{(\pi_i(t+1)+\epsilon_i(t))(\pi_j(t+1)+\epsilon_j(t))}\cdot |\pi_i(t+1)\epsilon_j(t)-\pi_j(t+1)\epsilon_i(t)|\\
&\leq \tilde C\cdot \big(| \epsilon_j(t)|+|\epsilon_i(t)|\big),
\end{align*}
for some constant $\tilde C>0$. Hence
\begin{equation}
\begin{split}
 \sum\limits_{t=0}^\infty\sqrt{t}\left(\max\limits_{i,j}|\pi_i(t+1)\delta_i(t)-\pi_j(t+1)\delta_j(t)|\right)
<\infty.
\end{split}
\end{equation}
The proof is finished.
\end{proof}
\begin{rem}\mbox{~}
\begin{itemize}
\item[(1)] We may set $\delta_i(t)=\frac{ct^\alpha}{\pi_i(t+1)}$ with $\alpha\in[-1,-1/2)$ and $c>0$ in Lemma \ref{L4.2} (3).
\item[(2)] When all $P(t)$ are doubly-stochastic, then $\pi(0)=\pi(1)=\dots=\frac{1}{N}1_{N}$. So we may simply set $\delta_i(t)=ct^\alpha$ with $\alpha\in[-1,-1/2)$ and $c>0$ in Lemma \ref{L4.2} (3).
\end{itemize}
\end{rem}

In the following proposition, we provide a sufficient condition $(\mathcal A1)'$ for $\{P(t)\}_{t\geq0}$ to satisfy $(\mathcal A1)$, described in terms of directed graphs. The condition $(\mathcal A1)'$ is adapted from \cite{NO2}. We mention that however, $(\mathcal A1)$ is not equivalent to $(\mathcal A1)'$. (See the remark that follows Proposition \ref{P4.4} below.)
We begin with the following definition. 
\begin{defn} For graph $G= (V, E)$ with $V = \{1, \cdots, N\}$ and $E \subset V \times V$, we call $G$ strongly connected if for any $(i,j) \in V\times V$, there exists a path from $i$ to $j$, i.e., there exists a finite sequence of vertices  $i=i_0,i_1,\dots,i_n=j$ in $V$ such that $(i_{l-1},i_l)\in E$ for $l=1,\dots,n$.
\end{defn}
\begin{prop}\label{P4.4} Given the sequence of $N\times N$ matrices $\{P(t)\}_{t \geq 0}$ with nonnegative entries, define the directed graphs $G(t)=(V,E(t))$ $(t\geq0)$ with $V=\{1,\dots,N\}$ and $E(t)\subseteq V\times V$ in the following way:
\[
(i,j)\in E(t)\quad\mbox{if and only if}\quad p_{ij}(t)>0.
\]
Then the following  assumption implies $(\mathcal A1)$:
\begin{enumerate}

\item[$(\mathcal A1)'$]  The sequence $\{P(t)\}_{t\geq0}$ of $N\times N$ row-stochastic matrices with positive diagonal entries satisfy
\begin{equation}\label{eq-4-11}
p^+:=\inf_{t\geq0}\left[\min_{i,j:~p_{ij}(t)>0}p_{ij}(t)\right]>0
\end{equation}
and there exists $t_0\in\mathbb N$ such that for all $t\geq0$, the union graph 
\[
\bigcup_{t\leq s< t+t_0}G(s):=\left(V,\bigcup_{t\leq s< t+t_0}E(s)\right)
\]
is strongly connected. 
\end{enumerate}
\end{prop}
\begin{proof}

 Define $A(t)=(a_{ij}(t))$ $(t\geq0)$ by
\[
a_{ij}(t)=\begin{cases}
1\quad\mbox{if}\quad (i,j)\in E(t) \quad\mbox{and}\quad i\neq j,\\
0\quad\mbox{otherwise.}
\end{cases}
\]
Then we have $P(t)\geq p^+ (I_N+A(t))$, for any $t\geq0$ by definition of $p^+$ given in \eqref{eq-4-11}, and so
\begin{equation}\label{eq-exm32-1}
\begin{split}
P(t+t_0,t)&\geq (p^+)^{t_0}(I_N+A(t+t_0-1))\dots(I_N+A(t))
\\
&\geq(p^+)^{t_0}\left(I_N+\sum_{k=0}^{t_0-1}A(t+k)\right)=: B(t,t_0).
\end{split}
\end{equation}
The $(i,j)$-entry of $B(t,t_0)$ is positive if and only if
\[
(i,j)\in\bigcup_{0\leq k\leq t_0-1}E(t+k)\quad \textrm{or}\quad i=j.
\]
By the strong connectivity of $(\mathcal A1)'$, the matrix $B(t,t_0)$ is an irreducible matrix, i.e., for any pair of indices $(i,j)$ there exists a finite sequence $i=i_0,i_1,\dots,i_n=j$ such that $(i_{m-1},i_{m})$-entry of the matrix $B(t,t_0)$ is positive for all $1\leq m\leq n$, and has positive diagonal entries. So, the same is true for $P(t+t_0,t)$ by \eqref{eq-4-1}. Noting that $P(t+(N-1)t_0,t)$ is equal to 
\begin{equation}
P(t+(N-1)t_0, t) = \prod_{k=0}^{N-1} P(t+(k+1)t_0, t+k t_0)
\end{equation}
which is a product of $N-1$ irreducible matrices with positive diagonal entries, we can see that proving the following claim would ensure the positivity of  $P(t+(N-1)t_0,t)$, thereby finishing the proof.
\end{proof}

\begin{lem}Each row of a product of $k$ irreducible matrices with positive diagonal entries has at least $k+1$ positive entries. $(1\leq k\leq N-1)$
\end{lem}
\begin{proof}
Let $A^{(\ell)}=(a_{ij}^{(\ell)})$ $(\ell=1,\dots,k)$ be such matrices, and consider the product $A^{(1)}\dots A^{(k)}$. We proceed by induction on $k$. \\
(i) $k=1$: For each $i$, we have $a_{ii}^{(1)}>0$, and the existence of $j\neq i$ satisfying $a_{ij}^{(1)}>0$ can be shown in the following way: Choose any $q\neq i$. Since $A^{(1)}$ is irreducible, there exists a finite sequence $i=i_0, i_1,\dots,i_n=q$ of indices such that $a_{i_{m-1}i_{m}}^{(1)}>0$ for all $1\leq m\leq n$. Set $j=i_{s}$, where $s$ is the minimal index satisfying $i_s\neq i$.\\
(ii) Suppose that the claim holds for $k-1\leq N-2,$ Set $A^{(1)}\dots A^{(\ell)}=(a_{ij}^{(1,\ell)})$. Note that if $a_{ij}^{(1,k-1)}$ is positive then so is $a_{ij}^{(1,k)}$, as can be seen in the following relation:
\begin{equation}\label{eq-exm32-2}
a_{ij}^{(1,k)}=\sum_{\ell}a_{i\ell}^{(1,k-1)}a_{\ell j}^{(k)}\geq  a_{ij}^{(1,k-1)}a_{j j}^{(k)}.
\end{equation}
Fix $i$. If the $i$-th row of $A^{(1)}\dots A^{(k-1)}$ had at least $k+1$ positive entries, then by \eqref{eq-exm32-2}, the same is true for  $A^{(1)}\dots A^{(k)}$. Now suppose that the $i$-th row of $A^{(1)}\dots A^{(k-1)}$ has exactly $k$ positive entries $(k \leq N-1)$.
To complete the induction step, it suffices to show that there exists $j$ with $a_{ij}^{(1,k)}>0=a_{ij}^{(1,k-1)}$. Pick any $\ell$ with $a_{i\ell}^{(1,k-1)}=0$. Since $A^{(k)}$ is irreducible, there exists a finite sequence $i=i_0, i_1,\dots,i_n=q$ of indices such that $a_{i_{m-1}i_{m}}^{(k)}>0$ for all $1\leq m\leq n$. Set $j=i_{s}$, where $s$ is the minimal index satisfying $a_{ii_s}^{(1,k-1)}=0$, whose existence is guaranteed by the fact that $a_{i\ell}^{(1,k-1)}=0$. Then
\[
a_{ii_s}^{(1,k)}=\sum_{p}a_{ip}^{(1,k-1)}a_{p i_s}^{(k)}\geq  a_{ii_{s-1}}^{(1,k-1)}a_{i_{s-1}i_s}^{(k)}>0.
\]
Hence the claim holds for $k+1$.
\end{proof}
We close this section with showing that the condition $(\mathcal A1)$ is strictly more general than the condition $(\mathcal A1)'$. For this, we consider time independent $P(t)\equiv P$, where $P$ is \textit{irreducible} and  \textit{aperiodic}. In other words, the corresponding directed graph $G(t)\equiv G=(V,E)$ (as described in Proposition \ref{P4.4}) is
\begin{enumerate}
\item [(i)] strongly connected: for any pair of vertices $(i,j)$, $G$ contains a path from $i$ to $j$, i.e., a finite sequence $i=i_0,i_1,\dots,i_n=j$ such that $(i_{m-1},i_{m})\in E$  for all $1\leq m\leq n$, and
\item [(ii)] of period 1: for each vertex $i\in V$, the greatest common divisor of the lengths of all paths from $i$ to $i$ is equal to 1,
\end{enumerate}
which is necessary and sufficient for the matrix $P$ to be  \textit{primitive}, i.e., $P^T>0$ for some $T\in\mathbb N$ (see \cite{Seneta}, for example). It implies that our example $P(t)\equiv P$ satisfies the condition $(\mathcal A1)$. In addition, if $P$ has at least one zero diagonal entry, then our example does not satisfy the positive diagonal condition of $(\mathcal A1')$. Here is an example of such $P$:

\[
P=\begin{pmatrix}
0 & 1 & 0 &0 \\
0& 0 & 1 & 0\\
\frac{1}{2}  & 0  & 0 & \frac{1}{2} \\
0 & 0 & 1 & 0 
\end{pmatrix}.
\]
\section{Application of the main theorem to several algorithms}
In this section, we apply the result of Theorem \ref{Tmain} to derive the convergence results for the four examples discussed in Section 2.
\subsection{Convergence theorem for Example \ref{E3.1}}

\begin{cor}\label{cor-5-1}
Let $\tilde X(t)$ be a solution to the following algorithm, originated from \cite{JKJJ}: 
\begin{equation*}
  \tilde x_i(t+1)=\sum_{j=1}^N\tilde p_{ij}(t) (\tilde x_j(t)- \tilde\delta_j(t)\tilde g_i(t)),\quad \tilde g_i(t)\in\partial f_i(\tilde x_i(t)),\quad 1\leq i\leq N,\quad t\geq0. 
\end{equation*}
Suppose that the following conditions $(\tilde{\mathcal A} 1),(\tilde{\mathcal A }2),(\tilde{\mathcal A }3)$ hold. 
\begin{itemize}
\item ($\tilde{\mathcal A1}$): The sequence $\{\tilde P(t)\}_{t\geq0}$ of $N\times N$ row-stochastic matrices and its backward products satisfy
\[
\tilde p^+:=\inf_{t\geq0}\left[\min_{i,j:~\tilde p_{ij}(t)>0}\tilde p_{ij}(t)\right]>0
\]
and
\[
\tilde P(t+T,t)>0\quad (\forall~t\geq0)
\]
for some $T\in\mathbb N$.

\item ($\tilde{\mathcal A}2$): Each $\tilde \delta_i(\cdot)$ is nonnegative,  and $\sum_{t=0}^\infty \left(\|\tilde \Delta(t)\|\sup_{\ell\geq t}\|\tilde \Delta(\ell)\|\right)<\infty.$
\item ($\tilde{\mathcal A}3$): The set of absolute probability vectors for $\{\tilde P(t)\}_{t\geq0}$, denoted by
\[
\tilde \pi(t)=(\tilde \pi_1(t),\dots,\tilde \pi_N(t))^\top\quad (t\geq0),
\]
satisfy
$\sum_{t=0}^\infty \|\tilde \Delta(t)\tilde \pi(t)\|=\infty$ and $\sum_{t=0}^\infty\sqrt{t}\left(\max_{i,j}|\tilde \pi_i(t)\tilde \delta_i(t)-\tilde \pi_j(t)\tilde \delta_j(t)|\right)<\infty$.
\end{itemize}
Then there exists some minimizer $x^*$ of $f$ such that $\lim\limits_{t\to\infty}\tilde x_i(t)=x^*$ for all $1\leq i\leq N$. 
\end{cor}
\begin{proof}
For $t \geq 0$ we set $x_i (2t) = \tilde{x}_i (t)$ and $x_i (2t+1) = \tilde{x}_i (t) - \tilde \delta_i (t) \tilde{g}_i (t)$. We also define
\begin{equation}\label{eq-5-2}
P(2t)\equiv I_N,\quad P(2t+1)=\tilde P(t),\quad \Delta(2t+1)\equiv O_N,\quad \Delta(2t)=\tilde\Delta(t)\quad (t\geq0)
\end{equation}
and
\begin{equation}\label{eq-5-3}
\delta_i (2t) = \delta_i (t) \quad \textrm{and}\quad \delta_i (2t+1) =0.
\end{equation}
Then $X(t) =[x_1 (t), \cdots, x_N (t)]^\top$ satisfies
\begin{equation}
X(t+1) = P(t) X(t) - \Delta (t) G (t).
\end{equation}
Using \eqref{eq-5-2} it is direct to check that $(\mathcal A 1)$ and $({\mathcal A} 2)$ are equivalent to $(\tilde{\mathcal A}1)$ and $(\tilde{\mathcal A}2)$ respectively.  Since $P(2t) = I_N$ and $P(2t+1) =\tilde{P}(t)$ we find that $\pi (2t+1) = \pi (2t) =\tilde{\pi}(t)$. Using this and \eqref{eq-5-3} we see that $(\mathcal A3)$ and $(\tilde{\mathcal A}3)$ are equivalent. 
Combining this with Theorem \ref{Tmain}, we get the desired result.
\end{proof}
\subsection{Convergence theorem for Example \ref{E3.2}}
\begin{cor}
Let $X(t)$ be a solution to the following algorithm, originated from \cite{MA}: 
\begin{equation*}
\begin{cases}
\displaystyle     x_i(t+1)=\sum_{j=1}^N  p_{ij}  x_j(t)- \frac{\delta(t)}{z_{ii}(t)}  g_i(t),\quad  g_i(t)\in\partial f_i(  x_i(t)),\\
\displaystyle  z_i(t+1)=\sum_{j=1}^N  p_{ij}  z_j(t),\quad z_i(0)=\bold e_i,\quad 1\leq i\leq N,\quad t\geq0. 
\end{cases}
\end{equation*}
Suppose that $P=(p_{ij})$ is row-stochastic and primitive. Choose $\delta(t)=ct^\alpha$ with $\alpha\in[-1,-1/2)$ and $c>0$. Then there exists some minimizer $x^*$ of $f$ such that $\lim\limits_{t\to\infty}x_i(t)=x^*$ for all $1\leq i\leq N$. 
\end{cor}
\begin{proof}
In Example \ref{E3.2}, we see that $x_i (t)$ satisfies \eqref{frame} with $P (t) \equiv P$ and $\delta_i(t) = \frac{\delta (t)}{z_{ii}(t)}$. The condition $(\mathcal A1)$ is satisfied for any primitive $P$, and we have $\pi(0)=\pi(1)=\dots=:\pi$. A standard theory of row-stochastic matrices yields that as $t\to\infty$, $Z(t)=P^t\to 1_{N}\pi$, i.e., $z_{ii}(t)\to\pi_i$ $(\forall~i)$ geometrically fast (see \cite{Seneta}). Hence if we choose $\delta(t)=ct^\alpha$ with $\alpha\in[-1,-1/2)$ and $c>0$, then $(\mathcal A 2,3)$ are satisfied. 
\end{proof}
\subsection{Convergence theorem for Example \ref{E3.3}}
\begin{cor}\label{cor-5-3}
Let $Y(t), W(t)$ be a solution to the following algorithm:
\begin{equation*}
\begin{cases}
\displaystyle  y_i(t+1)=\sum_{j=1}^N a_{ij}(t)y_j(t),\quad1\leq i\leq N, \quad t\geq0, \\
\displaystyle  w_i(t+1)=\sum_{j=1}^N a_{ij}(t)w_j(t)- ct^\alpha  g_i(t),\quad g_i(t)\in\partial f_i\left(\frac{w_i(t)}{y_i(t)}\right),
\end{cases}
\end{equation*}
with $\alpha\in[-1,-1/2)$ and $c>0$.
Suppose that $Y(0)>0$ and $({\mathcal A} 1*)$ hold. 
\begin{itemize}
\item ($\mathcal A1*$): The sequence $\{A(t)\}_{t\geq0}$ of $N\times N$ column-stochastic matrices with no zero rows satisfy
\[
a^+:=\inf_{t\geq0}\left[\min_{i,j:~a_{ij}(t)>0}a_{ij}(t)\right]>0
\]
and
\[
A(t+T,t):=A(t+T-1)\dots A(t+1)A(t)>0\quad (\forall~t\geq0)
\]
for some $T\in\mathbb N$.
\end{itemize}
Then there exists some minimizer $x^*$ of $f$ such that $\lim\limits_{t\to\infty} \frac{w_i(t)}{y_i(t)}=x^*$ for all $1\leq i\leq N$. 
\end{cor}
\begin{proof}
We may write the above scheme as
\begin{equation}\label{eq-5-4}
\frac{w_i (t+1)}{y_i (t+1)} = \sum_{j=1}^N \frac{a_{ij}(t)y_j(t)}{y_i (t+1)} \frac{w_j (t)}{y_j(t)} - \frac{ct^{\alpha}}{y_i (t+1)} g_i (t).
\end{equation}
We set $P(t) = \textrm{diag} (Y(t+1))^{-1} A(t) \textrm{diag}(Y(t))$ for $t \geq 0$. Then $P(t)$ is row-stochastic as checked in \eqref{eq-3-1}. We take $\delta_i(t)=\frac{ct^\alpha}{y_i(t+1)}$ with $\alpha\in[-1,-1/2)$ and $c>0$. 
Then $x_i (t) = \frac{w_i (t)}{y_i (t)}$ satisfies
\begin{equation}
x_i (t+1) = \sum_{j=1}^N p_{ij}(t) x_j (t) - \delta_i (t) g_i (t),\quad g_i (t) \in \partial f_i (x_i (t)).
\end{equation}
We aim to show that $P(t)$ and $\delta_i (t)$ satify the assumption $(\mathcal A 1)-(\mathcal A 3)$. To find the absolute probability vectors of $P(t)$, we note that
\begin{align*}
Y(t+1)^\top P(t) &=Y(t+1)^\top \operatorname{diag}(Y(t+1))^{-1} A(t)\operatorname{diag}(Y(t)) \\
&=(1_{N})^\top A(t) \operatorname{diag}(Y(t))\\
&=(1_{N})^\top \operatorname{diag}(Y(t))\\
&=Y(t)^\top.
\end{align*}
By multiplying $1_{N}$ to both sides, we also get 
\[
Y(t+1)^\top1_{N}=Y(t)^\top1_{N}=\dots=Y(0)^\top 1_{N}.
\]
Therefore, to show that $\pi(t)=\frac{1}{Y(0)^\top 1_{N}}Y(t)$, it only remains to prove that the condition ($\mathcal A1*$) combined with $Y(0)>0$ is a sufficient condition for $(\mathcal A1)$, which implies the uniqueness of the set of absolute probability vectors.
Indeed, we have
\[
A(t+T,t)>0\quad\Rightarrow\quad P(t+T,t)=\operatorname{diag}(Y(t+T))^{-1} A(t+T,t)\operatorname{diag}(Y(t))>0.
\]
The condition $p^+>0$ can be proved in the following way: for $t\geq T$, we have
\[
Y(t )=A(t ,t-T)Y(t-T)\geq (a^+)^{T} 1_{N}(1_{N})^\top Y(t-T)=\left((a^+)^{T} (1_{N})^\top Y(0)\right) 1_{N}
\]
and
\[
Y(t+1)\leq 1_{N} (1_{N})^\top Y(t+1) =\left( (1_{N})^\top Y(0)\right) 1_{N}
\]
so
\[
P(t)=\operatorname{diag}(Y(t+1))^{-1} A(t)\operatorname{diag}(Y(t))\geq {(a^+)^{T} A(t),}\quad t\geq T.
\]
Hence 
\[
p^+\geq\min\left\{\min_{0\leq t< T}\left[\min_{i,j:~a_{ij}(t)>0}a_{ij}(t)\right],{(a^+)^{T+1}  }\right\}>0.
\]
Finally, note that $\delta_i(t)=\frac{ct^\alpha}{y_i(t+1)}$ with $\alpha\in[-1,-1/2)$ and $c>0$, i.e., $\theta_i(t)= ct^\alpha$ satisfy the condition in Lemma \ref{L4.2} (3), thereby satisfying $(\mathcal A2),(\mathcal A3).$ Summing up, we obtain the desired result.
\end{proof}
Recall that in Example \ref{E3.3} we have used the substitution $x_i(t)=\frac{w_i(t)}{y_i(t)}$. We remark that we can get an analogous statement of Proposition \ref{P4.4} for $(\mathcal A1*)$ instead of $(\mathcal A1)$.
\subsection{Convergence theorem for Example \ref{E3.4}}
\begin{cor}
Let $\hat X(t), \hat Y(t), \hat W(t)$ be a solution to the following algorithm, originated from \cite{NO2}: 
\begin{equation*}
\begin{cases}
\displaystyle\hat  y_i(t+1)=\sum_{j=1}^N\hat a_{ij}(t)\hat  y_j(t) \in\mathbb R_+,\\
\displaystyle {\hat  x_i(t)=\hat  w_i(t)- c t^\alpha\hat g_i(t),\quad \hat  g_i(t)\in\partial f_i\left(\frac{\hat  w_i(t)}{\hat y_i(t)}\right),}\\
\displaystyle \hat  w_i(t+1)=\sum_{j=1}^N\hat a_{ij}(t)\hat  x_j(t) \in\mathbb R^d.
\end{cases}
\end{equation*}
with $\alpha\in[-1,-1/2)$ and $c>0$.
Suppose that $\hat Y(0)>0$ and $(\hat{\mathcal A} 1*)$ hold.
\begin{itemize}
\item ($\hat{\mathcal A}1*$): The sequence $\{\hat A(t)\}_{t\geq0}$ of $N\times N$ column-stochastic matrices with no zero rows satisfy
\[
\hat a^+:=\inf_{t\geq0}\left[\min_{i,j:~\hat a_{ij}(t)>0}\hat a_{ij}(t)\right]>0
\]
and
\[
\hat A(t+T,t):=\hat A(t+T-1)\dots \hat A(t+1)\hat A(t)>0\quad (\forall~t\geq0)
\]
for some $T\in\mathbb N$.
\end{itemize}
 Then there exists some minimizer $x^*$ of $f$ such that $\lim\limits_{t\to\infty} \frac{\hat w_i(t)}{\hat y_i(t)}=\lim\limits_{t\to\infty} \frac{\hat x_i(t)}{\hat y_i(t)}=x^*$ for all $1\leq i\leq N$. 
\end{cor}
\begin{proof}[Proof 1]
We let 
\begin{equation}
w_i (2t) =  \hat{w}_i (t),\quad w_i (2t+1) =  \hat{x}_i (t) \quad \textrm{and} \quad y_i(2t+1)=y_i(2t)=\hat y_i(t).
\end{equation}
We write down the first and the second lines of the algorithm as
\begin{equation*}
\begin{cases}
\displaystyle  y_i(2t+1)= y_i(2t),\quad1\leq i\leq N, \quad t\geq0, \\
\displaystyle  w_i(2t+1)= w_i(2t)- ct^\alpha  g_i(2t),\quad g_i(2t)\in\partial f_i\left(\frac{w_i(2t)}{y_i(2t)}\right).
\end{cases}
\end{equation*}
Next we write the first and the third lines of the algorithm as
\begin{equation*}
\begin{cases}
\displaystyle  y_i(2t+2)=\sum_{j=1}^N \hat a_{ij}(t)y_j(2t+1),\quad1\leq i\leq N, \quad t\geq0, \\
\displaystyle  w_i(2t+2)=\sum_{j=1}^N \hat a_{ij}(t)w_j(2t+1).
\end{cases}
\end{equation*}
Now we set $\theta_i (2t) =  c(2t)^{\alpha}  $, $\theta_i(2t+1)=0$, $A(2t) = I_N$, and $A(2t+1) = \hat A(t)$. 

Then we see that $W(t), Y(t)$ satisfies the algorithm \eqref{algorithm3}. By following the same argument as in the proof of Corollary \ref{cor-5-3}, we get the desired result. 
\end{proof}

\begin{proof}[Proof 2]
We let
\begin{equation}
\tilde x_i (t+1) = \frac{\hat{w}_i (t+1)}{\hat{y}_i (t+1)}.
\end{equation}
Then the scheme is written as
\begin{equation}
\begin{split}
\tilde x_i (t+1) & = \sum_{j=1}^N \frac{\hat{a}_{ij} (t)}{\hat{y}_i (t+1)}  \Big( \hat{y}_j (t) \tilde x_j (t) - ct^\alpha\tilde{g}_j (t)\Big)
\\
& = \sum_{j=1}^N \frac{\hat{a}_{ij} (t) \hat{y}_j (t)}{\hat{y}_i (t+1)} \Big( \tilde x_j (t) - \frac{ct^\alpha}{\hat{y}_j (t)} \tilde{g}_j (t)\Big), \quad \tilde  g_i(t)\in\partial f_i\left(\tilde x_i(t)\right)
\end{split}
\end{equation}
Thus we have
\begin{equation}
\tilde x_i (t+1)=: \sum_{j=1}^N \tilde{p}_{ij}(t) \Big(\tilde x_j (t) - \tilde\delta_j (t) \tilde  g_j(t)\Big), \quad \tilde  g_i(t)\in\partial f_i\left(\tilde x_i(t)\right)
\end{equation}
where $\tilde{p}_{ij} (t)= \frac{\hat{a}_{ij} (t) \hat{y}_j (t)}{\hat{y}_i (t+1)}$ and $\tilde\delta_j (t) = \frac{ct^\alpha}{\hat{y}_j (t)}$. It corresponds to the algorithm in Corollary \ref{cor-5-1}. Moreover, from the proof of Corollary \ref{cor-5-3} we see that $(\tilde{\mathcal A} 1)-(\tilde{\mathcal{A}} 3)$ are satisfied for $\tilde{P}$ and $\tilde{\Delta}$. Therefore we may apply Corollary \ref{cor-5-1} to obtain the convergence result. 
\end{proof}
\begin{proof}[Proof 3]
We let 
\begin{equation}
x_i (2t) = \frac{\hat{w}_i (t)}{\hat{y}_i (t)}\quad \textrm{and}\quad x_i (2t+1) = \frac{\hat{x}_i (t)}{\hat{y}_i (t)}.
\end{equation}
We write the second line of the algorithm as
\begin{equation}
\frac{\hat{x}_i (t)}{\hat{y}_i (t)} = \frac{\hat{w}_i (t)}{\hat{y}_i (t)} - \frac{ct^{\alpha}}{\hat{y}_i (t)} \hat{g}_i (t), \quad \hat{g}_i (t) \in \partial f_i \bigg( \frac{\hat{w}_i (t)}{\hat{y}_i (t)}\bigg).
\end{equation}
Now we set $\delta_i (2t) = \frac{ct^{\alpha}}{\hat{y}_i (t)} $, $g_i (2t) = \hat{g}_i (t)$, and $p_{ij}(2t) = \delta_{ij}$. Then the above equality is written as
\begin{equation}\label{eq-5-5}
x_i (2t+1) = x_j (2t) - \delta_i (2t) g_i (2t),\quad g_i (2t) \in \partial f_i (x_i (2t)).
\end{equation}
Next we write down the third line of the algorithm as
\begin{equation}
\frac{\hat{w}_i (t+1)}{\hat{y}_i (t+1)} = \sum_{j=1}^N \frac{\hat{a}_{ij} (t) \hat{y}_j (t)}{\hat{y}_i (t+1)} \frac{\hat{x}_j (t)}{\hat{y}_j (t)}.
\end{equation}
Let $p_{ij} (2t+1) =\frac{\hat{a}_{ij} (t) \hat{y}_j (t)}{\hat{y}_i (t+1)}$ and $\delta_{i} (2t+1) =0$. Then this is written in terms of $x_i$ as
\begin{equation}\label{eq-5-6}
x_{i}(2t+2) = \sum_{j=1}^N p_{ij}(2t+1) x_i (2t+1).
\end{equation}
From \eqref{eq-5-5} and \eqref{eq-5-6} we see that $x_i (t)$ satisfies \eqref{frame}. Now it remains to check that $P(t)=(p_{ij}(t))$ and $\Delta(t)=\operatorname{diag}(\delta_1(t),\dots,\delta_N(t))$ satisfy the conditions $(\mathcal A1), (\mathcal A2), (\mathcal A3)$. Note that 
\[
P(t) = \textrm{diag} (Y(t+1))^{-1} A(t) \textrm{diag}(Y(t)),
\]
where
\[
A(2t)=I_N,\quad A(2t+1)=\hat A(t), \quad Y(2t+1)=Y(2t)=\hat Y(t).
\]
By following the same argument as in the proof of Corollary \ref{cor-5-3}, we can see that $\pi(t)=\frac{1}{Y(0)^\top 1_{N}}Y(t)$, and thus $(\mathcal A1)-(\mathcal A3)$ hold. Thus we have the desired result. 
\end{proof}
\section{Proof of Theorem \ref{Tmain}}

In this section we give the proof of Theorem \ref{Tmain}. For this aim, we state and prove preliminary lemmas which are inspired by \cite{NO2}.
\begin{lem}\label{lem-6-1}Suppose $(\mathcal A1)$ hold. Assume that a sequence $\{b(t)\}_{t\geq0}$ of vectors in $\mathbb R^{N\times1}$ satisfy 
\[
\sup_{t\geq0}\|b(t)\|_1<\infty\quad\mbox{and}\quad b(t)^\top 1_{N}=0\quad (\forall ~t\geq0).
\]
Then the sequence $\{X(t)\}_{t \geq 0}$ generated by \eqref{frame2} satisfy the following: 
\begin{enumerate}
\item 
We have
\[
\left\|X(t)^\top b(t)\right\|_1\leq  C\lambda^{t}+C\sum_{k=0}^{t-1}\lambda^{t-k-1}\|\Delta(k) \|_\infty,
\]
for some constant $C>0$ independent of $t\geq1$.
\item If $\lim_{t\to\infty}\|\Delta(t)\|_\infty=0$ holds, then
\[
\lim_{t\to\infty}\left\|X(t)^\top b(t)\right\|_1=0.
\]
\item If $(\mathcal A2)$ holds, then
\[
\sum_{t=1}^\infty\|\Delta(t)\|_\infty\left\|X(t)^\top b(t)\right\|_1<\infty.
\]
\end{enumerate}
\end{lem}
\begin{proof}
(1) 
  Using \eqref{frame} iteratively, we have
\[
X(t)=P(t,0)X(0)-\sum_{k=0}^{t-1}P(t,k+1)\Delta(k)G(k),\quad t\geq1.
\]
 Multiplying $b(t)$ to both sides and using the triangle inequality, we deduce
\begin{equation}\label{eq-6-0}
\begin{split}
\left\| X(t)^\top b(t)\right\|_1&=\left\|X(0)^\top P(t,0)^\top b(t)-\sum_{k=0}^{t-1} G(k)^\top \Delta(k)^\top P(t,k+1)^\top b(t)\right\|_1\\
&\leq  \|X(0)^\top \|_1\| P(t,0)^\top b(t)\|_1+\sum_{k=0}^{t-1} \| G(k)^\top \|_1\|\Delta(k)^\top \|_1\|P(t,k+1)^\top b(t) \|_1
\\
&\leq  \|X(0)  \|_\infty\tau( P(t,0))\| b(t)\|_1+\sum_{k=0}^{t-1} \|G(k) \|_\infty\|\Delta(k) \|_\infty\tau(P(t,k+1))\| b(t) \|_1,
\end{split}
\end{equation}
where we used Lemma \ref{L2.4} in the second inequality. 
By Lemma \ref{L4.2} (1) we have $\tau(P(t,0)) \leq \tilde C \lambda^t$ for some $\tilde C>0$ and $0 < \lambda <1$. We also have
\[
\sup_{k\geq0}\|G(k) \|_\infty=\sup_{k\geq0}\max_{1\leq i\leq N}\|g_i(k)\|_1\leq \max_{1\leq i\leq N}L_i<\infty.
\]
 Using this we bound the right hand side of \eqref{eq-6-0} as 
\begin{equation}\label{eq-6-1}
\begin{split}
\left\| X(t)^\top b(t)\right\|_1&\leq C\lambda^{t}+C\sum_{k=0}^{t-1}\lambda^{t-k-1}\|\Delta(k) \|_\infty.
\end{split}
\end{equation}
(2) By (1), it suffices to show that 
\begin{equation}\label{C-5}
\lim_{t\to\infty}\sum_{k=0}^{t-1}\lambda^{t-k-1}\|\Delta(k)\|_\infty=0.
\end{equation}
For any $1\leq m\leq t-1$ we have
\begin{align*}
\sum_{k=0}^{t-1}\lambda^{t-k-1}\|\Delta(k)\|_\infty&\leq \sum_{k=0}^{t-1}\lambda^{t-k-1}\sup_{\ell\geq k}\|\Delta(\ell)\|_\infty\\
&\leq\Big(\sup_{\ell\geq 0}\|\Delta(\ell)\|_\infty\Big)\sum_{k=0}^{m-1}\lambda^{t-k-1}+\Big(\sup_{\ell\geq m}\|\Delta(\ell)\|_\infty\Big)\sum_{k=m}^{t-1}\lambda^{t-k-1}\\
&\leq \Big(\sup_{\ell\geq 0}\|\Delta(\ell)\|_\infty\Big)\frac{\lambda^{t-m}}{1-\lambda}+\Big(\sup_{\ell\geq m}\|\Delta(\ell)\|_\infty\Big)\frac{1}{1-\lambda}
\end{align*}
By plugging in $m=\lfloor\frac{t}{2}\rfloor$ and sending $t\to\infty$, we can see that \eqref{C-5} holds.\\\\
(3) We apply \eqref{eq-6-1} to obtain
\begin{align*}
&\sum_{t=1}^\infty\|\Delta(t)\|_\infty\left\|X(t)^\top b(t)\right\|_1\\
&\leq \sum_{t=1}^\infty\|\Delta(t)\|_\infty\left(C\lambda^{t}+C\sum_{k=0}^{t-1}\|\Delta(k)\|_\infty\lambda^{t-k-1}\right)\\
&=C\sum_{t=1}^\infty\|\Delta(t)\|_\infty\lambda^{t}+C\sum_{t=1}^\infty\sum_{k=0}^{t-1}\|\Delta(t)\|_\infty\|\Delta(k)\|_\infty\lambda^{t-k-1}. \end{align*}
We bound this as follows
\begin{align*}
&\sum_{t=1}^\infty\|\Delta(t)\|_\infty\left\|X(t)^\top b(t)\right\|_1\\
&\leq C\sum_{t=1}^\infty\Big(\sup_{\ell\geq1}\|\Delta(\ell)\|_\infty\Big)\lambda^{t}+C\sum_{t=1}^\infty \sum_{k=0}^{t-1}\Big(\sup_{\ell\geq k}\|\Delta(\ell)\|_\infty\Big)\|\Delta(k)\|_\infty\lambda^{t-k-1} \\
&= C\frac{\lambda}{1-\lambda}\Big(\sup_{\ell\geq1}\|\Delta(\ell)\|_\infty\Big) +C\sum_{k=0}^\infty \sum_{t=k+1}^{\infty}\Big(\sup_{\ell\geq k}\|\Delta(\ell)\|_\infty\Big)\|\Delta(k)\|_\infty\lambda^{t-k-1} \\
&= C\frac{\lambda}{1-\lambda}\Big(\sup_{\ell\geq1}\|\Delta(\ell)\|_\infty\Big)+C\frac{1}{1-\lambda}\sum_{k=0}^\infty  \Big(\sup_{\ell\geq k}\|\Delta(\ell)\|_\infty\Big)\|\Delta(k)\|_\infty <\infty,
\end{align*}
where we used $(\mathcal A 2)$ in the last inequality. The proof is done.
\end{proof}
\begin{lem}\label{lem-6-2}
Suppose $(\mathcal A 1)$-$(\mathcal A 2)$ holds. Then we have
 \[ \|X(t)\|=O(\sqrt{t}),\quad t\to\infty.
\]
\end{lem}
\begin{proof}
Using \eqref{frame} and the fact that $L_i =\sup_{z \in \mathbb{R}^d} \sup_{g \in \partial f_i (z)} \|g\|_1 <\infty$, we find
\begin{align*}
\|X(t+1)\|_\infty&=\|P(t)X(t)-\Delta(t)G(t)\|_\infty\\
&\leq\|P(t)\|_\infty \|X(t)\|_\infty+\|\Delta(t)\|_\infty\|G(t)\|_\infty \\
&=\|X(t)\|_\infty+\|\Delta(t)\|_\infty\|G(t)\|_\infty.
\end{align*}
Iterating gives us the following estimate:
\begin{align*}
 \|X(t)\|_\infty &\leq\|X(0)\|_\infty+\sum_{k=0}^{t-1}\|\Delta(k)\|_\infty\|G(k)\|_\infty\\
&\leq\|X(0)\|_\infty+\left(\max_{1\leq i\leq N}L_i\right)\sum_{k=0}^{t-1}\|\Delta(k)\|_\infty.
\end{align*}
Now we use the Cauchy-Schwarz inequality  to deduce
\begin{align*}
\|X(t)\|_\infty&\leq \|X(0)\|_\infty+\left(\max_{1\leq i\leq N}L_i\right)\left(\sum_{k=0}^\infty\|\Delta(k)\|_\infty^2\right)^\frac{1}{2}\sqrt{t} =O(\sqrt{t}),
\end{align*}
where $(\mathcal A 2)$ is used. The proof is done.
\end{proof}
\begin{lem}\label{lem-6-3}
Suppose $(\mathcal A 1)$-$(\mathcal A 2)$ holds. Then for any $u\in\mathbb R^{d\times1}$, we have
\begin{align*}
\left\|X(t+1)^\top \pi(t+1)-u\right\|_2^2&\leq\left\|X(t)^\top \pi(t)-u\right\|_2^2+\|\Delta(t)\|_\infty^2\left(\max_{1\leq i\leq N} L_i\right)^2\\
&+\frac{2}{N} \|\Delta(t) \pi(t+1)\|_1\sum_{i=1}^N \left(f_i(u)-f_i(X(t)^\top \pi(t))\right)\\
&+2 \left(\max_{i,j}|\pi_i(t+1)\delta_i(t)-\pi_j(t+1)\delta_j(t)|\right)\sum_{i=1}^NL_i\left\|X(t)^\top \pi(t) -u\right\|_2\\
&+4\|\Delta(t)\|_\infty\sum_{i=1}^N  L_i \left\|X(t)^\top(\pi(t)-\bold e_i)  \right\|_2.
\end{align*}
\end{lem}
\begin{proof} We recall from \eqref{eq-2-1} that $\pi (t)$ satisfies $\pi (t+1)^\top P(t) = \pi (t)^\top$. Combining this with \eqref{frame}, we find the following equality
\begin{align*}
 X(t+1)^\top \pi(t+1)-u &=X(t)^\top P(t)^\top\pi(t+1)-G(t)^\top \Delta(t)^\top \pi(t+1) -u\\
&=X(t)^\top  \pi(t)-G(t)^\top \Delta(t)^\top \pi(t+1) -u.
\end{align*}
Using this identitiy, we compute
\begin{align*}
\left\| X(t+1)^\top \pi(t+1)-u \right\|_2^2&=\left\| X(t)^\top  \pi(t)-G(t)^\top \Delta(t)^\top \pi(t+1) -u \right\|_2^2\\
&=\left\|X(t)^\top \pi(t)-u\right\|_2^2+ \left\|\sum_{i=1}^N\pi_i(t+1) \delta_i(t)g_i(t) \right\|_2^2\\
&\quad-2 \sum_{i=1}^N\pi_i(t+1)\delta_i(t)\left\langle X(t)^\top \pi(t)-u,g_i(t)\right\rangle\\
&=:\left\|X(t)^\top \pi(t)-u\right\|_2^2+I_1+I_2.
\end{align*}
We estimate $I_1$ as
\begin{align*}
I_1&\leq  \left(\sum_{i=1}^N\pi_i(t+1) \delta_i(t)\left\|g_i(t) \right\|_2\right)^2\\
&\leq \|\Delta(t)\|_\infty^2\left(\sum_{i=1}^N\pi_i(t+1) L_i\right)^2\\
&\leq \|\Delta(t)\|_\infty^2\left(\max_{1\leq i\leq N} L_i\right)^2,
\end{align*}
where $\sum_{i=1}^N \pi_i (t+1) =1$ is used in the third inequality. Next we decompose $I_2$ as follows:
\begin{align*}
I_2&=-2\sum_{i=1}^N\pi_i(t+1)\delta_i(t)\left\langle X(t)^\top\bold e_i -u,g_i(t)\right\rangle\\
&\quad-2 \sum_{i=1}^N\pi_i(t+1)\delta_i(t)\left\langle X(t)^\top (\pi(t)-\bold e_i) ,g_i(t)\right\rangle\\
&=:I_{21}+I_{22}.
\end{align*}
By the convexity of $f_i$, we obtain
\begin{align*}
I_{21}&=-2\sum_{i=1}^N\pi_i(t+1)\delta_i(t)\left\langle x_i(t)-u, g_i(t)\right\rangle\\
&\leq 2 \sum_{i=1}^N\pi_i(t+1)\delta_i(t)\left(f_i(u)-f_i(x_i(t))\right)\\
&= 2 \sum_{i=1}^N\left(\frac{1}{N}\sum_{j=1}^N\pi_j(t+1)\delta_j(t)\right) \left(f_i(u)-f_i(X(t)^\top \pi(t))\right)\\
&\quad+ 2 \sum_{i=1}^N\left(\pi_i(t+1)\delta_i(t)-\frac{1}{N}\sum_{j=1}^N\pi_j(t+1)\delta_j(t)\right)\left(f_i(u)-f_i(X(t)^\top \pi(t))\right)\\
&\quad+2 \sum_{i=1}^N\pi_i(t+1)\delta_i(t)\left(f_i(X(t)^\top \pi(t))-f_i(X(t)^\top \bold e_i)\right).
\end{align*}
Here we recall that $\sup_{g \in \partial f_i (z)} \|g\|_1 \leq L_i$ for all $z \in \mathbb{R}^d$ and achieve the following estimates:
\begin{align*}
I_{21}&\leq \frac{2}{N} \|\Delta(t) \pi(t+1)\|_1\sum_{i=1}^N \left(f_i(u)-f_i(X(t)^\top \pi(t))\right)\\
&\quad+2 \left(\max_{i,j}|\pi_i(t+1)\delta_i(t)-\pi_j(t+1)\delta_j(t)|\right)\sum_{i=1}^NL_i\left\|X(t)^\top \pi(t) -u\right\|_2\\
&\quad+2\|\Delta(t)\|_\infty\sum_{i=1}^N  L_i \left\|X(t)^\top(\pi(t)-\bold e_i) \right\|_2
\end{align*}
and
\begin{align*}
I_{22}&\leq 2\|\Delta(t)\|_\infty\sum_{i=1}^N  L_i \left\|X(t)^\top(\pi(t)-\bold e_i)  \right\|_2.
\end{align*}
Combining the above estimates finishes the proof.
\end{proof}
We recall from \cite[Lemma 7]{NO2} the following result. 
\begin{lem}\label{lem-6-4}
Consider a minimization problem $\min_{x\in\mathbb R^d}f(x)$, where $f:\mathbb R^d\to\mathbb R$ is continuous. Suppose that the solution set $X^*$ of the problem is nonempty, and let $\{x(t)\}_{t\geq0}$ be a sequence such that for all $x^*\in X^*$ and $t\geq0$,
\[
\|x(t+1)-x^*\|^2\leq (1+b(t))\|x(t)-x^*\|^2-a(t)(f(x(t))-f(x^*))+c(t),
\]
where 
\[
a(t),~b(t),~c(t)\geq0~(\forall~t\geq0),\quad \sum_{t=0}^\infty a(t)=\infty,\quad\sum_{t=0}^\infty b(t)<\infty,\quad\sum_{t=0}^\infty c(t)<\infty.
\]
Then the sequence $\{x(t)\}_{t\geq0}$ converges to some solution $x^*\in X^*$.
\end{lem}

We conclude this section by presenting the proof of the main theorem. 
\begin{proof}[Proof of Theorem \ref{Tmain}]

We let $x(t): = X(t)^\top \pi (t)$. Then the estimate of Lemma \ref{lem-6-3} with $u=x^*$ gives
\begin{equation}
\|x(t+1) -x^*\|^2 \leq \|x(t) -x^*\|^2 - a(t) (f(x(t))- f(x^*)) + c(t),
\end{equation}
where $a(t):=\frac{2}{N}\|\Delta(t)\pi(t+1)\|_1$ and
\begin{align*}
c(t)&:= \|\Delta(t)\|_\infty^2\left(\max_{1\leq i\leq N} L_i\right)^2\\
&+2 \left(\max_{i,j}|\pi_i(t+1)\delta_i(t)-\pi_j(t+1)\delta_j(t)|\right)\sum_{i=1}^NL_i\left\|X(t)^\top \pi(t) -x^*\right\|_2\\
&+4\|\Delta(t)\|_\infty\sum_{i=1}^N  L_i \left\|X(t)^\top(\pi(t)-\bold e_i)  \right\|_2.
\end{align*}
By $(\mathcal A 3)$ we see that
\begin{equation}\label{eq-6-2}
\sum_{t=0}^{\infty} a(t) = \frac{2}{N} \sum_{t=0}^{\infty}\|\Delta (t) \pi (t+1)\|_{1} =\infty.
\end{equation}
Next we estimate $\sum_{t=0}^{\infty} c(t)$. It follows from $(\mathcal A 2)$ that
\begin{equation}
\sum_{t=0}^{\infty} \|\Delta (t)\|_{\infty}^2 \leq \sum_{t=0}^{\infty} \Big( \|\Delta (t)\| \sup_{l \geq t}\|\Delta (t)\|\Big) < \infty.
\end{equation}
Combining Lemma \ref{lem-6-2} and $(\mathcal A 3)$, we derive
\begin{equation}
\begin{split}
&\sum_{t=0}^{\infty} 2 \left(\max_{i,j}|\pi_i(t+1)\delta_i(t)-\pi_j(t+1)\delta_j(t)|\right)\sum_{i=1}^NL_i\left\|X(t)^\top \pi(t) -x^*\right\|_2\\
&\leq 2C \sum_{t=0}^{\infty} \sqrt{t} \left(\max_{i,j}|\pi_i(t+1)\delta_i(t)-\pi_j(t+1)\delta_j(t)|\right) < \infty.
\end{split}
\end{equation}
Next we apply Lemma \ref{lem-6-1} (3) to find
\begin{equation}
\sum_{t=0}^{\infty} 4 \|\Delta (t)\|_{\infty} \sum_{i=1}^N L_i \|X(t)^\top (\pi (t) -\bold e_i)\|_2 < \infty,
\end{equation}
where we used that $(\pi (t) -\bold  e_i)^\top 1_{N} = 1 -1 =0$ for all $t \geq 0$ and $1\leq i \leq N$. Combining the above estimates we find that 
\begin{equation}\label{eq-6-3}
\sum_{t=0}^{\infty} c(t) < \infty. 
\end{equation}
Given the estimates \eqref{eq-6-2} and \eqref{eq-6-3}, we may apply Lemma \ref{lem-6-4} to conclude that $x(t) = X(t)^\top \pi (t)$ converges to some minimizer $x^* \in \mathbb{R}^d$. In addition, we have
\begin{equation}
\lim_{t \rightarrow \infty} \|X(t)^\top \pi (t) -x_i (t)\|_2 =\lim_{t \rightarrow \infty} \|X(t)^\top (\pi (t) -\bold e_i)\|_2 =0,
\end{equation}
which yields that $\lim_{t \rightarrow \infty} x_i (t) =x^*$. The proof is complete.
\end{proof}


\begin{thebibliography}{20}


\bibitem{AGL} M. Akbari, B. Gharesifard, T. Linder, Distributed online convex optimization on time-varying directed graphs. IEEE Trans. Control Netw. Syst. 4 (2017), no. 3, 417-428.

\bibitem{AS} S. A. Alghunaim, A. H. Sayed,Alghunaim,  Linear convergence of primal-dual gradient methods and their performance in distributed optimization. Automatica J. IFAC 117 (2020), 8 pp.


\bibitem{BBKW} A. S, Berahas, R. Bollapragada, N. S. Keskar, and E. Wei, Balancing Communication and Computation in
Distributed Optimization. IEEE Trans. Automat. Control 64 (2019),  3141–3155.

\bibitem{CS} S. Chatterjee, and E. Seneta, Towards consensus: some convergence theorems on repeated averaging. \textit{J. Appl. Prob.} {\bf14}, 89-97.

\bibitem{Dobrushin1} Dobrushin, R. L. Central Limit Theorem for Nonstationary Markov Chains. I. \textit{Theory Probab. Appl.} {\bf1}(1), 65-80 (1956).
\bibitem{Dobrushin2} Dobrushin, R. L. Central Limit Theorem for Nonstationary Markov Chains. II. \textit{Theory Probab. Appl.} {\bf1}(4), 329–383 (1956).

\bibitem{Jakovetic} D. Jakoveti\'c, A Unification and Generalization of Exact Distributed First-Order Methods, \textit{IEEE Transactions on Signal and Information Processing over Networks,} vol. 5, no. 1, pp. 31-46, 2019.

\bibitem{JKJJ} B. Johansson, T. Keviczky, M. Johansson, and K. H. Johansson, Subgradient methods and consensus algorithms for solving convex optimization problems. 2008 47th IEEE Conference on Decision and Control.

\bibitem{KDG} D. Kempe, A. Dobra, J. Gehrke, Gossip-based computation of aggregate information. In: Proceedings of the 44th Annual IEEE symposium on foundations of computer science, 482–491 (2003).


\bibitem{Kolmogoroff} A. Kolmogoroff, Zur Theorie der Markoffschen Ketten. \textit{Math. Ann.} {\bf112}, 155–160 (1936).

\bibitem{MA} Van Sy Mai and E. H. Abed, Distributed optimization over weighted directed graphs using row stochastic matrix. 2016 American Control Conference (ACC), Boston, MA, 2016, pp. 7165--7170,

\bibitem{MA2} Van Sy Mai and E. H. Abed, Distributed optimization over directed graphs with row stochasticity and constraint regularity.  
Automatica J. IFAC 102 (2019), 94--104.


\bibitem{NL} A. Nedi\'c, and J. Liu, On convergence rate of weight-averaging dynamics for consensus problems. \textit{IEEE Trans. Autom. Control}, {\bf62}(2), 766-781 (2017).

\bibitem{NO} A. Nedi\'c, and A. Ozdaglar, Distributed subgradient methods for multi-agent optimization. \textit{IEEE Trans. Autom. Control} , {\bf 54}(1), 48-61 (2009). 
\bibitem{NO2} A. Nedi\'c, and A. Olshevsky. Distributed optimization over time-varying directed graphs. \textit{IEEE Trans. Autom. Control}, {\bf 60}(3), 601-615 (2015).

\bibitem{NO3}A. Nedi\'c, A. Olshevsky, Stochastic gradient-push for strongly convex functions on time-varying directed graphs. IEEE Trans. Automat. Control 61 (2016), no. 12, 3936-3947.

\bibitem{Seneta} E. Seneta. \textit{Non-negative matrices and Markov chains,} 2nd Ed. Springer-Verlag New York (1981). 

\bibitem{SHL} A. Sundararajan, B. Hu, and L. Lessard, Robust convergence analysis of distributed optimization algorithms. In Proceedings of the 55th annual Allerton con- ference on communication control, and computing (Allerton) (pp. 1206–1212) (2017).

\bibitem{SXK} F. Saadatniaki, R. Xin, and U. A. Khan, Decentralized optimization over time-varying graphs with row and column-stochastic matrices. \textit{IEEE Trans. Autom. Control}, {\bf65}(11), 4769-4780 (2020).

\bibitem{TAHR} H. Taheri, M. Aryan, H. Hamed, P, Ramtin, Quantized Decentralized Stochastic Learning over Directed Graphs,   Proceedings of the 37th International Conference on Machine Learning, PMLR 119:9324-9333, 2020

\bibitem{TBA} J. Tsitsiklis, D. Bertsekas, and M. Athans, Distributed asynchronous deterministic and stochastic gradient optimization algorithms. \textit{IEEE Trans. Autom. Control}, {\bf 31}(9), 803-812 (1986).

\bibitem{XPNK} R. Xin, S. Pu, A. Nedi\'c, and U. A. Khan, A general framework for decentralized optimization with first-order methods, Proceedings of the IEEE, vol. 108, no. 11, pp. 1869--1889, (2020).


\bibitem{XZSX} J. Xu, S. Zhu, Y. C. Soh, and L. Xie, A Bregman splitting scheme for distributed optimization over networks. IEEE Transactions on Automatic Control, 63 (11), 3809–3824 (2018).



\bibitem{ZYGY} S. Zhang, X. Yi, J. George, and T. Yang, Computational convergence analysis of distributed optimization algorithms for directed graphs. In Proceedings of the 15th IEEE international conference on control and automation (ICCA) (2019).
\end{thebibliography}
\end{document}